\documentclass[10 pt, journal]{ieeetran}
\IEEEoverridecommandlockouts                                % This command is only needed if 
                                                         
\usepackage{mydef}
\usepackage{cases}
\pdfoutput=1
\title{\LARGE \bf
Linear Convergence of Primal-Dual Gradient Methods and their Performance in Distributed Optimization
}

\author{Sulaiman A. Alghunaim$^{*}$  and Ali H.~Sayed$^{\dagger}$, {\em Fellow}, IEEE% <-this % stops a space
\thanks{This work  was  supported  in  part  by   grant  205121-184999 from the Swiss National Science Foundation.}% <-this % stops a space
\thanks{$^{*}$S. A. Alghunaim is with the ECE Department, University of California at Los Angeles (UCLA). Email:{\tt\small salghunaim@ucla.edu}.
        }%
\thanks{$^{\dagger}$A. H. Sayed is with the Ecole Polytechnique Federale de Lausanne
EPFL, School of Engineering, CH-1015 Lausanne, Switzerland e-mail:
        {\tt\small ali.sayed@epfl.ch}.}%
}

\graphicspath{ {./CDC_images/} }
\begin{document}

%\address{$^{\star}$Department of Electrical and Computer Engineering,  University of California, Los Angeles \\
%$^{+}$School of Engineering, Ecole Polytechnique Federale de Lausanne, Switzerland}

\maketitle
\thispagestyle{plain}
\pagestyle{plain}
%\thispagestyle{empty}
%\pagestyle{empty}

\begin{abstract}
In this work, we revisit a  classical incremental implementation of the primal-descent dual-ascent gradient method used for the solution of equality constrained optimization problems. We provide a short proof that establishes the linear (exponential) convergence of the algorithm for smooth strongly-convex cost functions and study its relation to the non-incremental implementation.   We also  study the effect of the augmented Lagrangian penalty term on the performance of distributed optimization algorithms for the minimization of aggregate cost functions over multi-agent networks.
\end{abstract}

\begin{IEEEkeywords}
Primal-dual methods, linear convergence, Arrow-Hurwicz, augmented Lagrangian, distributed optimization. 
\end{IEEEkeywords}

%\IEEEpeerreviewmaketitle
\section{Introduction}
Consider the  constrained optimization problem:
\begin{align}
 \underset{w\in \mathbb{R}^M}{\text{minimize   }}& \quad
  J(w), \quad {\rm s.t.} \ Bw=b\label{glob1} 
\end{align}
where $J(w):\real^M \rightarrow \real$ is a smooth function assumed to satisfy Assumption \ref{assump:cost} further ahead, $B \in \real^{E \times M}$, and $b \in \real^E$.  Consider also the saddle point problem:
\eq{
\min_{w\in \mathbb{R}^M} \max_{\lambda \in \real^E} \ L_\rho(w,\lambda)  \label{saddle_problem}
}
where
\eq{
L_\rho(w,\lambda) \define J(w)+{\rho \over 2} \|Bw-b\|^2+\lambda\tran(Bw-b) \label{AL_rho}
}
is the augmented Lagrangian of problem \eqref{glob1},  $\lambda$ is a dual variable, and $\rho \geq 0$ is the augmented Lagrangian penalty parameter. Note that for $\rho=0$,  $L_0(w,\lambda)$ becomes the classical Lagrangian of problem \eqref{glob1}.   If a point $(w^\star,\lambda^\star)$ exists that solves  \eqref{saddle_problem},
then $w^\star$ is an optimal solution to the constrained problem when strong duality holds, which is the case under our assumptions \cite{boyd2004convex}.  {\color{black} A classical algorithm that solves \eqref{saddle_problem} is the  primal-dual (PD) gradient algorithm \eqref{alg_pd_da_inc}. In this algorithm,  $\grad J_\rho(w)$ denotes the gradient of $J_\rho(w)=J(w)+{\rho \over 2} \|Bw-b\|^2$ evaluated at $w$ and $(\mu_w,\mu_\lambda)$ are positive step-sizes (learning rates) chosen by the designer. The updates in \eqref{alg_pd_da_inc} are primal-descent dual-ascent steps applied to  \eqref{AL_rho} and it subsume the classical Lagrangian implementation  when $\rho=0$ and the augmented Lagrangian implementation when $\rho>0$.  Note that the updates in \eqref{alg_pd_da_inc} are {\em incremental} since the dual update \eqref{dual-ascent} uses the most recent primal variable $w_i$ and not $w_{i-1}$. If the dual update uses the previous primal iterate $w_{i-1}$, then we refer to the update as non-incremental.}

  This work provides a concise proof that establishes the linear convergence of recursion \eqref{alg_pd_da_inc} and studies its relation to the non-incremental implementation. We also study the effect of the penalty term 
 ${\rho \over 2} \|Bw-b\|^2$ on the performance of multi-agent consensus optimization algorithms.   Algorithms of the form \eqref{alg_pd_da_inc} have been applied in various applications including  wireless systems \cite{chen2012convergence}, power systems \cite{cherukuri2016initialization}, reinforcement learning \cite{macua2015distributed}, and network utility maximization \cite{feijer2010stability}. 
 
\begin{algorithm}[t] 
\caption*{\textrm{\bf{Algorithm}} (Incremental PD gradient method)}
{\bf Setting}:  Let $J_{\rho}(w)=J(w)+{\rho \over 2} \|Bw-b\|^2$ for some $\rho \geq 0$ and choose positive step-sizes $\mu_w$ and $\mu_\lambda$. Let $w_{-1}$ and $\lambda_{-1}$ be arbitrary initial conditions and repeat for $i \geq 0$ 
\begin{subequations}
\label{alg_pd_da_inc}
\eq{
w_i&=w_{i-1}-\mu_w \big(\grad J_{\rho}(w_{i-1})+ B\tran \lambda_{i-1}\big) \label{primal-descent} \\
\lambda_i &= \lambda_{i-1}+\mu_\lambda (Bw_i-b) \label{dual-ascent}
}
 \end{subequations}
\end{algorithm}

\subsection{Related Works}
 There exists a large body of literature on primal-dual saddle-point algorithms -- see \cite{arrow1958studies,kose1956solutions
 ,polyak1970iterative,bertsekas2014constrained,kallio1999large,
 nedic2009subgradient,feijer2010stability
 ,wang2011control,cherukuri2016asymptotic} and the references therein, including the seminal work \cite{arrow1958studies}, which proposed recursions of the type \eqref{alg_pd_da_inc} and established their convergence. These works focus on proving convergence to  an optimal solution without providing convergence rates, provide  sub-linear convergence rates (e.g.,  ${1 \over i}$ where $i$ is the iteration index), or show linear convergence from a starting point that is sufficiently close to a solution (local convergence). Some other works examined global linear convergence under different settings.

 The works \cite{niederlander2016distributed,dhingra2018proximal} focuses on {\em continuous} versions  of the primal-dual gradient dynamics and establish linear convergence for {\em augmented Lagrangian} implementations (i.e., they require the presence of the augmented Lagrangian term $\rho/2 \|Bw-b\|^2$, where $\rho$ is strictly positive). They also require $B$ to have full row rank. Similarly, the work \cite{qu2019exponential} establishes linear convergence for {\em continuous} primal-dual gradient dynamics for full row rank $B$, but it does not require the presence of the augmented Lagrangian term. Moreover, it was shown in \cite{qu2019exponential} that if the continuous dynamics is discretized using Euler discretization, then the discrete version converges linearly under small enough step sizes.  However, no upper bound is given on the step-sizes.  Moreover,  Euler discretization  uses identical step-sizes for the primal and dual updates (i.e., $\mu_w=\mu_\lambda$) and results in a {\em non-incremental} primal-dual dynamics.   Therefore, the results in \cite{niederlander2016distributed,dhingra2018proximal,qu2019exponential}  are not directly applicable to the discrete {\em incremental} implementation \eqref{alg_pd_da_inc} and do not provide clear bounds on the step-sizes.

{\color{black} 
We remark that linear convergence for various monotone operator methods have been established {\em albeit} under other conditions that are not satisfied in our setup. For example,  the linear convergence results in \cite{chen1997convergence} and \cite[Proposition 25.9]{bauschke2011convex} for forward-backward splitting methods  would require the saddle-point problem \eqref{saddle_problem} to be {\em both} strongly-convex with respect to $w$ and strongly-concave with respect to $\lambda$. This holds for example for problems with Lagrangian  $ L(w,\lambda) = J(w)+\lambda\tran Bw-g(\lambda)$ where  $J(w)$ and $g(\lambda)$ are both strongly-convex functions.  Similarly, the conditions used  in \cite{boct2015convergence,chambolle2016ergodic,Davis2017}  require the saddle-point problem \eqref{saddle_problem} to be strongly-convex with respect to $w$ and strongly-concave with respect to $\lambda$.  In our setup, $L_\rho(w,\lambda)$ is not strongly-concave with respect to $\lambda$. 

  The work \cite{du2018linear} showed that for saddle point problems with $ L(w,\lambda) = J(w)+\lambda\tran Bw-g(\lambda)$, linear convergence is possible without requiring the Lagrangian to be both strongly-convex and strongly-concave. In particular, it established linear convergence when the primal function $J(w)$ is smooth and convex, the dual function $-g(\lambda)$ is smooth and strongly-concave, and the additional assumption that $B$ is a {\em full column rank} matrix. Unlike the current work, the algorithm analyzed in \cite{du2018linear} is {\em non-incremental}; moreover, particular fixed step-sizes are needed to establish linear convergence --  \cite[~Theorem 3.1]{du2018linear}.}

Now, in the distributed optimization literature, various incremental primal-dual gradient algorithms have been proposed to solve multi-agent consensus optimization problems  -- see \cite{ling2015dlm,
chang2015multi,nedic2017achieving,
yuan2019exactdiffI,jakovetic2019unification} and references therein, which are mostly based on AL formulations. They have been shown to achieve linear convergence under strong-convexity even though the consensus constraint matrix is not full rank.  However, the analysis techniques used to establish the convergence of these methods either depend on the particular consensus constraint matrix and/or {\color{black} require the AL term to be strictly positive.   Unlike these works, our analysis does not require $\rho$ to be strictly positive.} Moreover, due to our unified Lagrangian and AL framework, we clarify the effect of the AL penalty term on the performance of these types of distributed algorithms.  Note that  the work \cite{towfic2015stabilitypd} studied  non-incremental primal-dual methods with identical step-sizes for {\em quadratic} distributed optimization.   It was found  in \cite{towfic2015stabilitypd}  that unlike AL methods, Lagrangian methods suffer from stability issues when the individual costs are not strongly-convex. Unlike \cite{towfic2015stabilitypd},  we study the affect of the AL penalty on the convergence rate of distributed algorithms.
 \subsection{Contribution}
 Given the above, this work has two main contributions:  {\bf I)} Through an original proof,  we establish the linear convergence of the {\em incremental} implementation \eqref{alg_pd_da_inc}. {\color{black} Moreover, we  show how the non-incremental implementation is related to the incremental one and establish its linear convergence while providing explicit  upper bounds on the step-sizes. Our proof technique does not require the AL parameter to be strictly positive nor do we require $B$ to have full row rank. }   {\bf II)}    We show the effect of the AL penalty term on the performance of distributed multi-agent optimization algorithms. {\color{black} Depending on the condition number of the agents' costs, we provide scenarios where the AL term is beneficial and other scenarios where it is not beneficial. }

 {\bf Notation and Terminology:} For a matrix $A \in \real^{M \times N}$, $\sigma_{\max}(A)$ denotes the maximum singular value  of $A$, $\sigma_{\min}(A)$ denotes the minimum singular value  of $A$, and  $\underline{\sigma}(A)$ denotes the smallest {\em non-zero} singular value. For a vector $x \in \real^M$ and  a positive constant $c>0$, we let $\|x\|_c^2$ denote the weighted norm $c\|x\|^2$. {\color{black} For any positive semidefinite matrix $A \in \real^{M \times M}$ the square root $A^{1 \over 2}$ is the solution of $X^2=A$. A function $f(x): \real^M \rightarrow \real$ is $\delta$-smooth if $
\|\grad f(x)-\grad f(y)\| \leq \delta \|x-y\|$ for any $x,y$ and some $\delta>0$. A smooth function $f(x)$ is $\nu$-strongly-convex if $(x-y)\tran \big(\grad f(x)-\grad f(y)\big) \geq \nu \|x-y\|^2$ for any $x,y$ and some $\nu>0$.}
\section{Auxiliary Results}
This section gives the auxiliary results leading to the main convergence result. We start with the following condition on the cost function.{\color{black}
\begin{assumption} \label{assump:cost}
{ {\bf (Cost function)}: It is assumed that a unique solution $w^\star$ exists for problem \eqref{glob1} and the cost function $J(w)$ is convex. It is also assumed that  $J(w)$ is  $\delta$-smooth, consequently, $J_{\rho}(w)=J(w)+{\rho \over 2} \|Bw-b\|^2$ is $\delta_\rho$-smooth with $\delta_\rho = \delta+\rho \sigma^2_{\max}(B)$. 
Moreover, the cost $J_{\rho}(w)$ is $\nu_{\rho}$-strongly-convex with respect to $w^\star$, namely,
\eq{
\hspace{-2mm} (x-w^\star)\tran \big(\grad J_\rho(x)-\grad J_\rho(w^\star)\big) \geq \nu_{\rho} \|x-w^\star\|^2, \quad \forall \ x  \label{stron-convexity}
} 
\noindent The  scalars satisfy  $0<\nu_{\rho} \leq \delta_{\rho}$ for any $\rho \geq 0$.  }\qd
\end{assumption} 
\begin{remark}[\sc Strong-convexity]{\rm
 If $J(w)$ is $\nu$-strongly-convex, then $w^\star$ is unique \cite[Example 5.4]{boyd2004convex} and condition \eqref{stron-convexity} will be satisfied with $\nu_\rho=\nu$. We remark that condition \eqref{stron-convexity} does not necessarily imply that $J(w)$ is strongly-convex w.r.t. $w^\star$ unless $\rho=0$. This condition is used instead of typical strong-convexity to be consistent with the conditions used to study the effect of the augmented Lagrangian term on the performance of distributed algorithms in Section \ref{section:dist}. } 
\qd
\end{remark}
}
\noindent  It is known that a pair $(w^\star,\lambda^\star)$ is an optimal solution to \eqref{saddle_problem} if, and only if, it satisfies the optimality conditions \cite{boyd2004convex}:
\begin{subequations} \label{KKT}
\eq{
\grad J(w^\star)+B\tran \lambda^\star&=0 \label{kkt1}\\
Bw^\star-b&=0 \label{kkt2}
}
\end{subequations}
  From \eqref{kkt1} and uniqueness of $w^\star$,  $\lambda^\star$ will be unique if $B$ has full row rank. In general $\lambda^\star$ is not necessarily unique. Motivated by \cite{shi2014onthe}, we will characterize a particular dual solution that we later show convergence to. For that result and later analysis, we need the following result.
\begin{lemma} \label{lemma:lower_bound}
{  If $\lambda_x$ is in the range space of $B \in \real^{E \times M}$, then it holds that:
\eq{
 \|B\tran \lambda_x\|^2 \geq \underline{\sigma}^2(B)\|\lambda_x\|^2 \label{lower_B_bound}
}
}
\end{lemma}
\begin{proof}   Introduce the truncated singular value decomposition \cite{Laub2004} of the positive semi-definite matrix $B\tran B =U_r \Sigma_r U_r\tran$, where $U_r \in \real^{M \times r}$ ($r$ denotes the rank of $B\tran B$) with $U_r\tran U_r=I_r$ and $\Sigma_r>0$ is a diagonal matrix with entries equal to the non-zero eigenvalues of $B\tran B$ ( i.e., the squared non-zero singular values of $B$). Since $\lambda_x$ is in the range space of $B$, it holds that  $\lambda_x=Bx$ for some $x$.     {\color{black} Thus, if we let  $u=\Sigma^{1 \over 2}_r U_r\tran x$, then 
\eq{
\hspace{-2mm} \|B\tran \lambda_x\|^2&=\|B\tran B x\|^2 =x\tran U_r\Sigma^2_r U_r\tran x \nonumber \\
 &=u\tran \Sigma_r u \geq \underline{\sigma}^2(B) \| u\| ^2=\underline{\sigma}^2(B) x\tran U_r\Sigma_r U_r\tran x 
\label{lower_bound_steps}}
The result follows since $x\tran U_r\Sigma_r U_r\tran x=\|\lambda_x\|^2$. The inequality follows since $\underline{\sigma}^2(B)$ is the smallest eigenvalue (or diagonal entry) of $\Sigma_r$ -- see \cite[Appendix A.5.2]{boyd2004convex}.}
\end{proof}
%\noindent This result will be used for the proof of the next lemma, which is motivated by the ADMM based works \cite{shi2014linear,ling2015dlm}. It will also be used to establish linear convergence, as will be explained after the next lemma next.  
\begin{lemma} \label{lemma:unique_dual}
 {\bf (Particular dual $\lambda^\star_b$)}: There exists a unique optimal dual variable, denoted by $\lambda^\star_b$,  lying in the range space of $B$.  
\end{lemma}

\begin{proof} The argument is motivated by \cite{shi2014onthe}.  Any solution $\lambda^\star$ of the linear system of equations given in \eqref{kkt1} can be decomposed into two parts $\lambda^\star=\lambda^\star_b+\lambda_n^\star$, where $\lambda^\star_b \in {\rm Range}(B)$ and $\lambda_n^\star \in {\rm Null}(B\tran)$ -- see \cite{Laub2004}. Therefore, if $(w^\star,\lambda^\star)$ satisfies \eqref{KKT}, then $(w^\star,\lambda^\star_b)$ also satisfies  \eqref{KKT}. We now show $\lambda^\star_b$ is unique by contradiction. Assume we have two {\em distinct} dual solutions $\lambda^\star_{b_1}= Bx_1$ and $\lambda^\star_{b_2}=Bx_2$ lying in the range space of $B$. Then, substituting into  \eqref{kkt1} and subtracting, we get $B\tran B(x_1-x_2)=0$.  It follows that $\|B(x_1-x_2)\|^2=0$ and, consequently, $B(x_1-x_2)=0$. This means that $\lambda^\star_{b_1}=Bx_1=Bx_2=\lambda^\star_{b_2}$, which is a contradiction.  
\end{proof}

\noindent   {\color{black} Note that if $\lambda_{i-1}$ belongs to the range space of $B$ (i.e., $\lambda_{i-1}=B x$ for some $x$) or  $\lambda_{i-1}=0$, then from $b=Bw^\star$ and  \eqref{dual-ascent} we know that $\lambda_i = \lambda_{i-1}+\mu_\lambda (Bw_i-b)=B \big(x+\mu_\lambda (w_i-w^\star)\big)$ will remain in the range space of $B$.}  Thus, $\{\lambda_{i}\}_{i \geq 0}$ will always remain in the range space of $B$ if $\lambda_{-1}$ belongs  to the range space of $B$ or $\lambda_{-1}=0$. This observation will  allow us to utilize the bound \eqref{lower_bound_steps} to  establish linear convergence to the particular saddle-point $(w^\star,\lambda^\star_b)$ without requiring a rank condition on the matrix $B$. 
\section{LINEAR CONVERGENCE RESULT}
 We are now ready to establish our main result. Let $\tw_i\define w_i-w^\star$ and $\tlam_i \define \lambda_i - \lambda^\star_b$ denote the primal and dual errors, respectively. 
\begin{theorem}{ {\bf (Linear convergence)}: \label{theorem1}
Let Assumption \ref{assump:cost} holds and assume the step-sizes are positive and satisfy:
\eq{
\mu_w  < {1  \over \delta_\rho}, \quad
 \mu_\lambda \leq {\nu_{\rho} \over \sigma^2_{\max}(B) } \label{step_sizes}
 }
 If $\lambda_{-1}=0$, then algorithm \eqref{alg_pd_da_inc} converges linearly to the particular saddle-point $(w^\star,\lambda^\star_b)$, namely, it holds that
\eq{
\|\tw_i\|_{c_w}^2+ \|\tlam_i\|_{c_\lambda}^2 
&\leq \gamma \big(\|\tw_{i-1}\|_{c_w}^2+\|\tlam_{i-1}\|_{c_\lambda}^2\big) 
\label{theorem_exp}}
where    $c_\lambda  >0$, 
$c_w=1-\mu_w \mu_\lambda \sigma^2_{\max}(B)>0$, and
$$\gamma \define \max \left\{1-\mu_w \nu_{\rho} (1-\mu_w \delta_\rho ),1-\mu_w \mu_\lambda \underline{\sigma}^2(B) \right\} <1$$
 } 
\end{theorem} 
\begin{proof}
  Subtracting $w^\star$ and $\lambda^\star_b$ from both sides of \eqref{alg_pd_da_inc} and using the optimality conditions \eqref{KKT} we get the coupled error recursion:
\begin{subequations}
\eq{
\hspace{-1mm }\tw_i&=\tw_{i-1}-\mu_w \big(\grad J_\rho(w_{i-1})-\grad J_\rho(w^\star)+ B\tran \tlam_{i-1}\big) \label{error_primal}  \\
\hspace{-1mm }\tlam_i &= \tlam_{i-1}+\mu_\lambda B\tw_i \label{error_dual}
}
\end{subequations}
 Squaring both sides of \eqref{error_primal} and \eqref{error_dual} we get
\eq{
\hspace{-1mm } \|\tw_i\|^2&=\|\tw_{i-1}-\mu_w \big(\grad J_\rho(w_{i-1})-\grad J_\rho(w^\star)\big)\|^2 \nonumber \\
& \ -2 \mu_w  \tlam_{i-1}\tran B \left(\tw_{i-1}-\mu_w\big( \grad J_\rho(w_{i-1})- \grad J_\rho(w^\star)\big)\right) \nonumber \\ 
& \quad    +\mu_w^2 \| B\tran \tlam_{i-1}\|^2
\label{er_sq_primal}}
and
\eq{
\hspace{-1mm } \|\tlam_i\|^2 &= \|\tlam_{i-1}\|^2+\mu_\lambda^2 \| B\tw_i\|^2 + 2 \mu_\lambda \tlam_{i-1}\tran B \tw_i \nonumber \\
&\overset{\eqref{error_primal}}{=} \|\tlam_{i-1}\|^2+\mu_\lambda^2\| B\tw_i\|^2 - 2 \mu_\lambda \mu_w   \|B\tran \tlam_{i-1}\|^2 \nonumber \\ 
& \quad +2 \mu_\lambda   \tlam_{i-1}\tran B \left(\tw_{i-1}-\mu_w\big( \grad J_\rho(w_{i-1})- \grad J_\rho(w^\star)\big)\right) \label{er_sq_dual}
}
Using the bound $\| B\tw_i\|^2 \leq \sigma^2_{\max}(B) \| \tw_i\|^2$, multiplying equation \eqref{er_sq_dual} by $c_\lambda\define \mu_w / \mu_\lambda$ and adding to \eqref{er_sq_primal} gives:
\eq{
\|\tw_i\|_{c_w}^2+ \|\tlam_i\|_{c_\lambda}^2 
  & \leq \|\tw_{i-1}-\mu_w \big(\grad J_\rho(w_{i-1})-\grad J_\rho(w^\star)\big)\|^2 \nonumber \\
& \quad + \|\tlam_{i-1}\|_{c_\lambda}^2-   \mu_w^2   \|B\tran \tlam_{i-1}\|^2 \label{err_sum}
}
where $c_w\define 1-\mu_w \mu_\lambda \sigma^2_{\max}(B)$. Note that from Lemma \ref{lemma:unique_dual}, $\lambda^\star_b$ lies in the range space of $B$. Moreover, since $\lambda_{-1}=0$, then we know that $\tlam_{i}$ will always lie in the range space of $B$. Thus, from \eqref{lower_B_bound} it holds that $
\|B\tran \tlam_{i-1}\|^2 \geq 
\underline{\sigma}^2(B) \|\tlam_{i-1}\|^2 $. Using this bound  in \eqref{err_sum}, we get:
\eq{
\|\tw_i\|_{c_w}^2+ \|\tlam_i\|_{c_\lambda}^2 
&  \leq  \|\tw_{i-1}-\mu_w \big(\grad J_\rho(w_{i-1})-\grad J_\rho(w^\star)\big)\|^2  \nonumber \\
& \quad +\big(1-\mu_w \mu_\lambda \underline{\sigma}^2(B)\big)\|\tlam_{i-1}\|_{c_\lambda}^2  \label{err_square}
} 
Since $J_\rho(w)$ is  $\delta_\rho$-smooth, it holds that \cite[Theorem 2.1.5]{nesterov2013introductory}:
\eq{\scalemath{0.95}{
\|\grad J_\rho(w_{i-1})-\grad J_\rho(w^\star)\|^2 	 \leq \delta_\rho \tw_{i-1}\tran \big(\grad J_\rho(w_{i-1})-\grad J_\rho(w^\star) \big) 
}\label{cocococococ} 
} 
Thus
\eq{
&\|\tw_{i-1}-\mu_w \big(\grad J_\rho(w_{i-1})-\grad J_\rho(w^\star)\big)\|^2  \nonumber \\
& \ \leq \big(1-\mu_w \nu_{\rho} (2-\mu_w \delta_\rho )\big) \|\tw_{i-1}\|^2 
}
for $\mu < 2 /\delta_\rho$. This follows directly by expanding the square and using the bounds \eqref{stron-convexity} and \eqref{cocococococ}. Let $\gamma_1=1-\mu_w \nu_{\rho} (1-\mu_w \delta_\rho )$. Since $c_w= 1-\mu_w \mu_\lambda \sigma^2_{\max}(B)$, it holds that:
\eq{
&\big(1-\mu_w \nu_{\rho} (2-\mu_w \delta_\rho )\big) \|\tw_{i-1}\|^2 =\gamma_1 \|\tw_{i-1}\|^2  -\mu_w \nu_{\rho} \|\tw_{i-1}\|^2 \nnb
 &=\gamma_1 \|\tw_{i-1}\|_{c_w}^2  -\mu_w( \nu_{\rho}-\mu_\lambda \sigma^2_{\max}(B) \gamma_1) \|\tw_{i-1}\|^2 \nnb
 &\leq \gamma_1 \|\tw_{i-1}\|_{c_w}^2 
}
where the last step we used the fact that the second term is non-positive  under the conditions $\mu_w < {1 \over \delta_\rho}$ and $\mu_{\lambda} \leq \nu_{\rho} / \sigma^2_{\max}(B)$. We conclude that equation \eqref{theorem_exp} holds by using  the previous two equations in \eqref{err_square}. Note that  for positive step-sizes it holds that $c_\lambda= {\mu_w \over \mu_\lambda}>0$. Moreover,  $c_w=1-\mu_w \mu_\lambda \sigma^2_{\max}(B)>0$ and $0<1-\mu_w \mu_\lambda \underline{\sigma}^2(B)<1$ if  $\mu_w \mu_\lambda < {1 \over \sigma^2_{\max}(B)}$.   This condition is satisfied under condition \eqref{step_sizes} because under these conditions we have  
 \eq{
 \mu_w \mu_\lambda < { \nu_{\rho} \over \delta_\rho \sigma^2_{\max}(B)}  \leq {1 \over \sigma^2_{\max}(B)} \nonumber
 }
  where the last inequality hold because $ \nu_{\rho}  \leq \delta_\rho$. 
  \end{proof}
   {\color{black}
 \noindent  Theorem \ref{theorem1} shows that under conditions \eqref{step_sizes}, the incremental algorithm \eqref{alg_pd_da_inc} converges linearly. We will show how to utilize this result to establish the linear convergence of the classical non-incremental (Arrow-Hurwicz) method \cite{arrow1958studies}. 
  \section{Non-incremental PD Gradient Method}  
  Consider  the non-incremental update (Arrow-Hurwicz):
 \begin{subnumcases}{\label{alg_pd_da_non_inc}}
w_i =w_{i-1}-\mu_w \big(\grad J_\eta(w_{i-1})+ B\tran \lambda'_{i-1}\big) \label{primal-descent_non} \\
\lambda'_i = \lambda'_{i-1}+\mu_\lambda (Bw_{i-1}-b) \label{dual-ascent_non}
 \end{subnumcases} 
 where $J_\eta(w) \define
  J(w)+{\eta \over 2} \|Bw-b\|^2$ and $\eta \geq 0$.  Different from \eqref{alg_pd_da_inc}, recursion \eqref{alg_pd_da_non_inc} uses  $w_{i-1}$ in the dual update instead of $w_i$. We will see that  these two different implementations are equivalent for particular choices of $\eta$ and $\rho$.
 \begin{lemma} {\bf (Equivalence of \eqref{alg_pd_da_inc} and \eqref{alg_pd_da_non_inc})}\label{lemma:noninc_equivalnt} { 
The primal iterates of the non-incremental recursion \eqref{alg_pd_da_non_inc}  are equivalent to the primal iterates of the incremental recursion \eqref{alg_pd_da_inc} if $\eta=\rho +\mu_\lambda$ and $\lambda'_{-1}=\lambda_{-1}-\mu_\lambda (Bw_{-1}-b)$.
}
\end{lemma}
\begin{proof} 
    Let $\eta=\rho+ \mu_\lambda$.  It holds that $
  J_\eta(w) =J_\rho(w) + {\mu_\lambda \over 2} \|Bw-b\|^2$ so that  $\grad J_\eta(w)=\grad J_\rho(w) + \mu_\lambda B\tran (Bw-b)$. Thus, for $\eta=\rho+ \mu_\lambda$ step \eqref{primal-descent_non} can be rewritten as: 
  \eq{
  w_i
  &=w_{i-1}-\mu_w \big(\grad J_{\rho}(w_{i-1})+ B\tran [\lambda'_{i-1}+\mu_\lambda (Bw_{i-1}-b)]\big) \nnb
  &=w_{i-1}-\mu_w \big(\grad J_{\rho}(w_{i-1})+ B\tran \lambda_{i-1}\big)
  }
  where we introduced the  change of variable $\lambda_i\define \lambda'_{i}+\mu_\lambda (Bw_i-b)$. Adding $\mu_\lambda (Bw_{i}-b)$ to both sides of \eqref{dual-ascent_non} and using $\lambda_i \define \lambda'_{i}+\mu_\lambda (Bw_i-b)$, we can directly rewrite \eqref{dual-ascent_non} as in \eqref{dual-ascent}.
  Thus, the primal iterates of recursion \eqref{alg_pd_da_non_inc} are equivalent to the primal iterates of recursion \eqref{alg_pd_da_inc} if $\lambda'_{-1}=\lambda_{-1}-\mu_\lambda (Bw_{-1}-b)$.  
  \end{proof}

\noindent    Lemma \ref{lemma:noninc_equivalnt} implies that the non-incremental implementation \eqref{alg_pd_da_non_inc} is an instance of the incremental implementation with $\rho=\eta-\mu_\lambda$. Recall that in algorithm \eqref{alg_pd_da_inc} we assume that $\rho \geq 0$. Therefore, if $\eta =\rho + \mu_\lambda \geq \mu_\lambda$,  the linear convergence of \eqref{alg_pd_da_non_inc} follows from Theorem \ref{theorem1} with $\rho =\eta - \mu_\lambda \geq 0$.   The case $0 \leq \eta < \mu_\lambda$ implies that $\rho =\eta - \mu_\lambda<0$. This case can also be analyzed using the exact same technique as in Theorem \ref{theorem1}.  To show that, it suffices to consider the classical case $\eta=0$.
 \begin{corollary}{\bf (Non-Incremental $\eta=0$)} \label{Cor_noninc_convergence}
  If the cost $J(w)$ is $\delta$-smooth and $\nu$-strongly-convex and the step-sizes satisfy:
\eq{
\mu_w < {1 \over \delta -\mu_\lambda
\sigma^2_{\min}(B)}, \quad
 \mu_\lambda \leq {\nu \over  2 \sigma^2_{\max}(B) } \label{step_sizes_noninc}
 }
 Then, recursion \eqref{alg_pd_da_non_inc} with $\eta=0$
  converges linearly  to the optimal saddle-point if  $\lambda'_{-1}=0$.  
 \end{corollary}
\begin{proof} See Appendix \ref{appendix_corro_noninc}.
 \end{proof}

By relating recursion \eqref{alg_pd_da_non_inc} to \eqref{alg_pd_da_inc}, we are able to establish its linear convergence and provide explicit upper bounds on the step-sizes as well. The works \cite{qu2019exponential} and \cite{du2018linear} also established the linear convergence of the non-incremental recursion \eqref{alg_pd_da_non_inc} with $\eta=0$. However, these works do not provide explicit upper bounds on the step-sizes \cite{qu2019exponential} or require particular fixed step-sizes to establish their result \cite{du2018linear}. 
   \begin{remark}[\sc Forward-Backward  Method]{\rm
Assume $b=0$ and consider  the forward-backward gradient algorithm \cite{komodakis2015playing}:
 \begin{subnumcases}{\label{alg_for-back}}
w_i =w_{i-1}-\mu_w \big(\grad J(w_{i-1})+ B\tran \lambda'_{i-1}\big) \label{fb_primal-descent} \\
\lambda'_i = \lambda'_{i-1}+\mu_\lambda B (2w_{i}-w_{i-1}) \label{fb-dual-ascent}
 \end{subnumcases}  
 By using a change of variable trick, the analysis of \eqref{alg_for-back} directly follows from Theorem \ref{theorem1} with $\rho=\mu_\lambda$.   In particular, by adding and subtracting $\mu_w \mu_\lambda B\tran B w_{i-1}$ to the R.H.S. of \eqref{fb_primal-descent}, letting $\lambda_i \define \lambda'_{i}-\mu_\lambda Bw_i$, and rearranging \eqref{fb-dual-ascent},  recursion \eqref{alg_for-back} can be equivalently written as  recursion \eqref{alg_pd_da_inc} ($b=0$) with $\rho=\mu_\lambda$.  } 
\qd
\end{remark}
  }
\section{Application: Distributed Optimization} \label{section:dist}
In this section, we study the benefit of the AL penalty term for distributed consensus optimization problems. 

Consider a network of $K$ agents that are connected through some network and interested in the following problem:
\begin{align}
 \underset{w\in \mathbb{R}^M}{\text{minimize   }}& \quad
 {1 \over K} \sum_{k=1}^K J_k(w) \label{distributed1} 
\end{align}
where $J_k(w):\real^{M} \rightarrow \real$ is a local cost function associated with agent $k$. In order to derive the algorithm that solves \eqref{distributed1} in a distributed manner, we will rewrite \eqref{distributed1} in an equivalent constrained form. We introduce a combination matrix $A=[a_{sk}]$ associated with the network. The entry $a_{sk}$ is the weight used by agent $k$ to scale information arriving from agent $s$ with $a_{sk}=0$ if $s$ is not a direct neighbor of agent $k$, i.e., there is no edge connecting them. 
{\color{black} \begin{assumption} \label{aaump-network}
The network is  static, undirected, and the matrix $A$ is assumed to be primitive, i.e., there exists some integer $p>0$ such that all entries of  $A^p$ are positive. We also assume $A$ to be symmetric, and doubly stochastic. \qd
\end{assumption}
}
			 There exists many rules to chose $A$ such as the Metropolis rule -- see \cite{sayed2014nowbook}, which satisfy Assumption \ref{aaump-network} as long as the network is connected. Under this assumption, it holds that  $I_K-A$ is positive semi-definite and $(I_K-A)x=0$ if, and only, if $x=c \one_K$ for any $c \in \real$ -- see \cite{yuan2019exactdiffI}. Therefore, if we let  $w_k \in \real^M$ denote a local copy of $w$ available at agent $k$ and introduce the network quantities:
\eq{
\sw &\define {\rm col}\{w_1,\cdots,w_K\} \in \real^{KM} \\
\cB &\define  (I_K-A)^{1 \over 2} \otimes I_M \label{combination}, \quad
\cJ(\sw) \define  \sum_{k=1}^K J_k(w_k)
}
 Then, it holds that $\cB \sw=0$ if, and only, if $w_k=w_s \ \forall \ k,s$ -- see \cite{yuan2019exactdiffI}. Thus, problem \eqref{distributed1} is equivalent to the following constrained problem:
\begin{align}
 \underset{\ssw\in \mathbb{R}^{KM}}{\text{minimize   }}& \quad
   \cJ(\sw) , \quad {\rm s.t.} \ \cB \sw=0\label{distributed2} 
\end{align}
 A direct application of \eqref{alg_pd_da_inc} to  problem \eqref{distributed2}  gives:
\begin{subequations} \label{pd_da_non}
\eq{
\sw_i&=\sw_{i-1}-\mu_w \grad_{\ssw} \cJ_\rho(\sw_{i-1})-\mu_w \cB \lambda_{i-1} \label{primal-descent_dist_not} \\
\lambda_i &= \lambda_{i-1}+\mu_\lambda \cB \sw_i \label{dual-ascent_dist_not}
}
 \end{subequations}
 where $\cJ_\rho(\sw) \define \cJ(\sw)+ {\rho \over 2} \| \cB \sw\|^2$ with $\rho \geq 0$. 
  Recursion \eqref{pd_da_non} is not distributed yet because $\cB$ need not have the network structure. However, this can be easily handled by a change of variable. Let $\sy_i=\cB \lambda_i$ and multiply \eqref{dual-ascent_dist_not} by $\cB$ gives:
 \begin{subequations} \label{pd_da_distributed}
\eq{
\sw_i&=\sw_{i-1}-\mu_w \grad_{\ssw} \cJ_\rho(\sw_{i-1})-\mu_w \sy_{i-1} \label{primal-descent_dist} \\
\sy_i &= \sy_{i-1}+\mu_\lambda \cB^2 \sw_i \label{dual-ascent_dist}
}
 \end{subequations}
Since  $\cB^2=(I_K-A) \otimes I_M$ has the network structure, then the $k$-th block of $\cB^2 \sw_i={\rm col}\{u_{k,i}\}_{k=1}^K$  has the distributed form $u_{k,i}=w_{k,i}-\sum_{s \in \cN_k} a_{sk}w_{s,i}$ where $\cN_k$ denotes the neighbors of agent $k$, including agent $k$. Therefore, recursion \eqref{pd_da_distributed} is distributed and agent $k$ can locally update its corresponding $k$-th blocks in $\sw_i$ and $\sy_i$.   
 \subsection{Relation to Other Algorithms} \label{section:relation to extra}
  Before we establish convergence of recursion \eqref{pd_da_distributed} and show the influence of the AL penalty term on its performance, we show how the derivation of  recursions \eqref{pd_da_non} and \eqref{pd_da_distributed} are related to some state of the art algorithms. 
 \subsubsection{EXTRA \cite{shi2015extra}}
Note that the saddle point interpretation of EXTRA  appeared in the work \cite{mokhtari2016dsa}. If we choose $\mu_w=\mu$, $\mu_\lambda={1 \over 2\mu}$, and $\rho={1 \over 2 \mu}$ in algorithm   \eqref{pd_da_non} we get:
  \begin{subequations} \label{extra_non}
\eq{
\sw_i&=\bar{\cA} \sw_{i-1}-\mu \grad_{\ssw} \cJ(\sw_{i-1}) - \mu \cB \lambda_{i-1} \label{extra_dist_not1} \\
\lambda_i &= \lambda_{i-1}+ {1 \over 2 \mu }\cB \sw_i \label{extra_dist_not2}
}
 \end{subequations}
 where $\bar{\cA}\define I-{1 \over 2}\cB^2 ={1 \over 2}(I+\cA)$ and $\cA=A \otimes I_M$.  By eliminating the dual-variable (see, e.g.,  \cite{jakovetic2019unification}), the above algorithm can be shown to be equivalent to the EXTRA algorithm in \cite{shi2015extra}, which requires communicating the primal variable once per iteration. 
 \subsubsection{Exact diffusion \cite{yuan2019exactdiffI}}
Consider the following update:
 \begin{subequations} \label{exdiff}
\eq{
\sw_i&=\bar{\cA} \bigg(\sw_{i-1}-\mu \grad_{\ssw} \cJ(\sw_{i-1})\bigg) - \mu \cB \lambda_{i-1} \label{exdiff1} \\
\lambda_i &= \lambda_{i-1}+ {1 \over 2 \mu } \cB \sw_i \label{exdiff2}
}
 \end{subequations}
which differs from EXTRA \eqref{extra_non} in the primal update where the gradient is also multiplied by $\bar{\cA}$. By eliminating the dual-variable, the above algorithm can be shown to be equivalent to the exact-diffusion algorithm from  \cite{yuan2019exactdiffI}. Different from a traditional gradient primal-descent \eqref{extra_dist_not1} that was used to derive EXTRA, exact diffusion uses incremental gradient descent steps -- see \cite{yuan2019exactdiffI} for details.   Exact diffusion  enjoys wider step-size  $\mu$ stability range and better convergence performance compared to EXTRA -- see \cite{yuan2019exactdiffII}.

 It is worth mentioning that if we consider the penalized unconstrained problem $\min_{\ssw} \ \cJ_\rho(\sw) = \cJ(\sw)+ {\rho \over 2} \| \cB \sw\|^2$ and apply two incremental gradient descent steps for the two terms in the penalized cost  with step-size $\mu$ and $\rho={1 \over \mu}$, we arrive at:
\begin{subequations} \label{diff1}
\eq{
\ssz_i&=\sw_{i-1}-\mu \grad_{\ssw} \cJ(\sw_{i-1}) \\
\sw_i&=   \cA \ssz_i 
}
\end{subequations}
which is the diffusion algorithm \cite{sayed2014nowbook,yuan2019exactdiffI}. The bias that arises from solving the penalized problem, rather than the original problem, can be corrected by employing exact diffusion \cite{yuan2019exactdiffI}.
 \subsubsection{DIGing \cite{nedic2017achieving}}
  If we choose a different penalty function $\cJ_\rho(\sw) = \cJ(\sw)+ {\rho \over 2} \|  \sw\|_{I-\cA^2}^2$ and set $\cB^2 \leftarrow \cB$  in algorithm   \eqref{pd_da_non} with $\mu_w=\mu$, $\mu_\lambda={1 \over \mu}$, and $\rho={1 \over  \mu}$ we get:
   \begin{subequations} \label{diging}
\eq{
\sw_i&=\cA^2 \sw_{i-1}-\mu \grad_{\ssw} \cJ(\sw_{i-1}) - \mu \cB^2 \lambda_{i-1} \label{diging1} \\
\lambda_i &= \lambda_{i-1}+{1 \over  \mu} \cB^2 \sw_i \label{diging2}
}
 \end{subequations}
 By eliminating the dual variable, this algorithm can be shown to be equivalent to DIGing -- see \cite[Section 2.2]{nedic2017achieving}.    We see that  the main difference from the EXTRA derivation is in the choice of the constraint and penalty matrices.  
 
{\color{black}
\subsubsection{Linearized ADMM \cite{ling2015dlm}}
 Consider an instance\footnote{We let $\tilde{d}_k=2cd_k+\rho=d$ in the DLM from \cite{ling2015dlm}. } of the decentralized linearized ADMM (DLM) method from \cite{ling2015dlm}:
 \begin{subequations} \label{lin_admm}
\eq{
\sw_i&=\sw_{i-1}-{1 \over d} \big( \grad_{\ssw} \cJ(\sw_{i-1})+c \cL \sw_i+ \sy_{i-1} \big) \label{admm1} \\
\sy_i &= \sy_{i-1}+c \cL \sw_i \label{admm2}
}
 \end{subequations}
 where $d,c>0$. The matrix $\cL$ is the oriented Laplacian matrix chosen such that the $k$-th block of $\cL \sw_i$ is equal to $\sum_{s \in \cN_k} w_{k,i}-w_{s,i}$. Recursion \eqref{lin_admm} is equivalent to \eqref{pd_da_distributed} with $\cB^2$ replaced by $\cL$, $\mu_\lambda=\rho=c$, and $\mu_w=1/d$.
 \begin{remark}[\sc Generalized Framework]{\rm
Based on the previous derivations, one can rewrite problem \eqref{distributed2} more generally as
 \begin{align}
 \underset{\ssw\in \mathbb{R}^{KM}}{\text{minimize   }}& \quad
   \cJ(\sw)+ {1 \over 2  } \| \sw\|_{\bar{\cC}}^2 , \quad {\rm s.t.} \ \cC \sw=0 \label{extra_problem} 
\end{align}
where   $\cC$ and $\bar{\cC}$ are general consensus matrices satisfying $\cC \sw=0$ if, and only, if $\bar{\cC} \sw=0$ if, and only, if $w_1=\cdots=w_K$. Various algorithms can be derived by proper choices of $\cC$ and $\bar{\cC}$ and using more general primal-dual algorithms.  For works focusing on unifying distributed algorithms, we refer interested readers to  \cite{jakovetic2019unification,alghunaim2019decentralized}.
 } 
\qd
\end{remark}
\begin{remark}[\sc Augmented Lagrangian Term]{\rm
We notice that most state-of-the-art algorithms are based on {\em augmented Lagrangian} formulations (i.e., they require $\rho$ to be strictly positive). However,  it is unclear whether the AL term is always beneficial. Unlike previous works, we reveal the influence of AL penalty term on convergence rate of distributed algorithms compared to the classical Lagrangian case ($\rho=0$).  } 
\qd
\end{remark}
}
\subsection{AL Penalty Term Influence}
To reveal the influence of the AL penalty parameter on the performance of distributed algorithms, we study the linear convergence properties of \eqref{primal-descent_dist}--\eqref{dual-ascent_dist}. To do that, we  let $(\sw^\star,\lambda^\star_b)$ be the point satisfying the optimality conditions of problem \eqref{distributed2} where  $\lambda^\star_b$ lies in the range space of $\cB$.  {\color{black} First, we recall the following result from \cite[Proposition 3.6]{shi2015extra}. 
 \begin{lemma}[\sc AL Penalized Cost] \label{lemma_penalized cost}
{  
		Let $\rho>0$. If each cost $J_k(w)$ is convex and $\delta$-smooth, and the aggregate cost ${1 \over K} \sum_{k=1}^K J_k(w)$  is  $\bar{\beta}$-strongly convex, then the penalized augmented cost $ \cJ(\sw)+ {\rho \over 2} \|  \sw\|_{\cB^2}^2$  is $\nu_{\rho}$-strongly-convex with respect to $\sw^\star$ where
  \eq{
 \hspace{-2mm} \nu_{\rho} \hspace{-0.5mm}\define  \hspace{-.5mm}\min\left\{\bar{\beta}-2\delta \eta,{\rho\underline{\sigma}^2(\cB) \eta^2 \over 4(\eta^2+1)}\right\}>0, \ \text{for $\eta \hspace{-0.5mm} \in \hspace{-0.5mm} \left(\hspace{-0.5mm} 0,{\bar{\beta} \over 2 \delta}\right)$} \label{nurho_definition}
  }
  and $\nu_\rho \rightarrow \bar{\beta}$ as $\rho \rightarrow \infty$. 
	} \qd
\end{lemma} 
 Note that even if the aggregate cost  ${1 \over K} \sum_{k=1}^K J_k(w)$   is strongly-convex, each cost $J_k(w_k)$ is not necessarily strongly-convex, e.g., $J_k(w)=\big(w(k)\big)^2$ where $w(k)$ is the $k$-th entry of $w \in \real^{M}$, is not strongly-convex with respect to $w \in \real^K$ but ${1 \over K} \sum_{k=1}^K J_k(w)={1 \over K} \|w\|^2$ is strongly-convex. The previous Lemma allows us to reveal the effect of the AL term through the following result. 
\begin{corollary}{  \label{theorem2}
  Assume that each cost $J_k(w)$ is convex and $\delta$-smooth and let Assumption \eqref{aaump-network} hold. Then, the following  result holds:
  \begin{itemize}
  \item If $\rho>0$, the aggregate cost ${1 \over K} \sum_{k=1}^K J_k(w)$  is  $\bar{\beta}$-strongly convex, and $\mu_w  < {1  \over \delta_\rho}, \quad
 \mu_\lambda \leq {\nu_\rho \over \sigma^2_{\max}(\cB) }$, then recursion \eqref{primal-descent_dist}--\eqref{dual-ascent_dist} with $\sy_{-1}=0$ converges linearly and the convergence rate is upper bounded by: 
 \eq{
\hspace{-2mm} \gamma_{AL} = \max \left\{1-\mu_w \nu_{\rho} (1-\mu_w \delta_\rho ),1-\mu_w \mu_\lambda \underline{\sigma}^2(\cB) \right\} \label{gama_L}
 }
 where $\delta_\rho=\delta+\rho \sigma^2_{\max}(\cB)$ and $\nu_\rho >0$ is defined in \eqref{nurho_definition}.
 \item If $\rho=0$, each cost $J_k(w)$ is $\beta_k$-strongly-convex, and $\mu_w  < {1  \over \delta}, \quad
 \mu_\lambda \leq {\nu_0 \over \sigma^2_{\max}(\cB) }$, then recursion \eqref{primal-descent_dist}--\eqref{dual-ascent_dist} with $\sy_{-1}=0$ converges linearly and the convergence rate is upper bounded by: 
 \eq{
\hspace{-2mm} \gamma_L = \max \left\{1-\mu_w \nu_{0} (1-\mu_w \delta ),1-\mu_w \mu_\lambda \underline{\sigma}^2(\cB) \right\} \label{gama_L}
 }
 where $\nu_0 = \min_k \beta_k$. 
  \end{itemize}
 } 
\end{corollary}

\begin{proof}  See Appendix \ref{appendix_corro_distributed} 
\end{proof}
  From the previous result, we see that for $\rho>0$, we only require the aggregate cost ${1 \over K} \sum_{k=1}^K J_k(w)$ to be strongly-convex to establish linear convergence since from Lemma \ref{lemma_penalized cost}, we know that for a strongly-convex aggregate cost ${1 \over K} \sum_{k=1}^K J_k(w)$, the penalized augmented cost $\cJ_\rho(\sw)$ is guaranteed to be strongly-convex w.r.t. $\sw^\star$. However, for the linear convergence of the case $\rho=0$, we require the stronger condition that each individual cost is strongly-convex. This is because the cost  $\cJ(\sw) \in \real^{MK} \rightarrow \real$ is  strongly-convex if, and only, if  each individual cost is strongly-convex -- see the argument in the proof of Corollary \ref{theorem2}.  Thus, the AL term is beneficial if the aggregate cost is strongly-convex but the individual costs are not -- see simulation section. However,  if each individual cost $J_k(w)$ is $\beta_k$-strongly-convex, then the presence of the AL term ($\rho>0$) can either degrade the performance compared to $\rho=0$ or improve the performance as we now explain.

  From the step-size conditions in Corollary \ref{theorem2}, the convergence rates $\gamma_L$ and $\gamma_{AL}$ have the form
 \eq{
 \gamma_L=1-c / \kappa_{L}, \quad \gamma_{AL}=1-c / \kappa_{AL}
 }
     for some $0<c<1$ where $\kappa_{L} \define \delta / \nu_0$ and $\kappa_{AL} \define \delta_\rho / \nu_\rho$ are the condition numbers of $\cJ(\sw)$ and $\cJ_\rho(\sw)$. 
       Note that $\nu_{\rho} \approx \bar{\beta}$ (for large enough $\rho$). If the condition number of the aggregate cost is much smaller than the condition number of the individual costs (e.g., $\bar{\beta} >>  \min_k \beta_k $), then the  AL method will  have faster convergence rate since $\kappa_{AL} < \kappa_{L}$, consequently $\gamma_{AL} < \gamma_{L} <1$. However, when the individual costs are well conditioned (e.g,  $\beta_k \approx \bar{\beta}$), then $\kappa_{AL} \approx \kappa_{L}$ and the AL penalty term is not that beneficial. Moreover, for large $\rho$ we can have $\kappa_{AL} > \kappa_{L}$; hence $\gamma_{L} < \gamma_{AL}$ and AL term  slows down the convergence rate.    
 }
\section{Simulation} 
 To illustrate the influence of the AL term on the performance of distributed algorithms, we consider the distributed optimization problem \eqref{distributed1} with quadratic costs $J_k(w)=w\tran R_k w + r_k\tran w$ where $w \in \real^{20}$, $R_k \in \real^{20 \times 20}$, and $r_k \in \real^{20}$.  We randomly generated a network of $K=20$ agents shown in the right side of Fig. \ref{fig:simresults1}. The matrix $A$ is generated using the Metropolis rule \cite{sayed2014nowbook}. Each vector $r_k$ is randomly generated with its entries uniformly selected between $[0,2]$. Note that the condition number of the cost $J_k(w)=w\tran R_k w + r_k\tran w$ is the ratio of the largest and smallest eigenvalues of $R_k$.  In our simulations, we construct the matrix $R_k$ under three different scenarios: 
  {\color{black}
\subsubsection{Well conditioned costs $J_k(w)$}
   The matrix $R_k$ is a randomly generated diagonal matrix with integer diagonal entries, each chosen between $[6,8]$. In this case, each $J_k(w)$ is well conditioned  because $8/6$ is not very large. The result for this scenario is shown on the left plot of Fig. \ref{fig:simresults1}. In all results, PD distributed refers to \eqref{pd_da_distributed} (with $\rho=0$) and AL PD distributed refers to \eqref{pd_da_distributed} with $\rho>0$,  EXTRA algorithm from \cite{shi2015extra}, and exact diffusion from \cite{yuan2019exactdiffI}.  The  step-sizes are manually chosen to get the best possible convergence rate for each algorithm. We notice that for this case, increasing $\rho$ decreases the performance compared to the case $\rho=0$.    In this scenario, we do not see any advantages of AL methods compared to the Lagrangian method ($\rho=0$) due to the reasons mentioned in the previous section. Note that EXTRA \eqref{extra_non} and exact diffusion \eqref{exdiff} converges slower since they require $\rho = 1/2\mu$, which cannot be tweaked independently from the step-size $\mu$.
\subsubsection{Ill conditioned costs $J_k(w)$}
 We now construct $R_k$ so that the local costs become  ill-conditioned. To do that,  we let $R_k$ to be a diagonal matrix where the $(k,k)$-th diagonal entry for each agent ($R_k(k,k)$) are chosen randomly between $[2,8]$ and the other diagonal entries are chosen uniformly between $(0,1)$. In this case, the ratio of the largest diagonal entry and the smallest can be very large making each $J_k(w)$  ill-conditioned. However, the aggregate cost ${1 \over K}\sum_{k=1}^K (w\tran R_k w + r_k\tran w)$ is better conditioned compared to the individual costs. This is because from our construction, the condition number  of $R= \sum_{k=1}^K R_k$ is smaller than the condition number of $R_k$.  The left plot  of Fig. \ref{fig:simresults2} shows the result for this case. The  step-sizes are manually chosen to get the best possible convergence rate for each algorithm. In this case, we see that the Lagrangian method performs poorly compared to AL PD method, EXTRA, and Exact diffusion.
 \subsubsection{Non-convex costs $J_k(w)$}
  We now consider the case where the individual costs $J_k(w)$ are non-convex but the aggregate cost $\sum_{k=1}^K J_k(w)$ is strongly-convex. To do that, we let $R_k$ be a diagonal matrix with the $(k,k)$-th diagonal entry for each agent, $R_k(k,k)$,  chosen randomly between $[2,8]$, the entries $R_k(k-1,k-1)=-{R_{k-1}(k-1,k-1)/2}$ for all $k \geq 2$. In this case, the individual costs $\{J_k(w)\}_{k \geq 2}$ are non-convex since they have negative diagonal entries.  However,  the aggregate cost ${1 \over K}\sum_{k=1}^K (w\tran R_k w + r_k\tran w)$ is strongly convex since $R=\sum_{k=1}^K R_k$ is positive-definite from construction.  The result of this set-up is shown in the right plot  of Fig. \ref{fig:simresults2}. The  step-sizes are manually chosen to get the best possible convergence rate for each algorithm. We see that the AL based methods still converge linearly. However, the PD distributed method diverges even under small step-sizes. This is because the cost $\cJ(\sw)=\sum_{k=1}^K (w_k\tran R_k w_k + r_k\tran w_k)$ is non-convex since the Hessian $\grad^2\cJ(\sw)={\rm blkdiag}\{R_k\}_{k=1}^K$ is indefinite. In contrast, the cost $\cJ_\rho(\sw)$ is strongly-convex for large $\rho$.
\begin{figure}[h]
	\includegraphics[width=\linewidth]{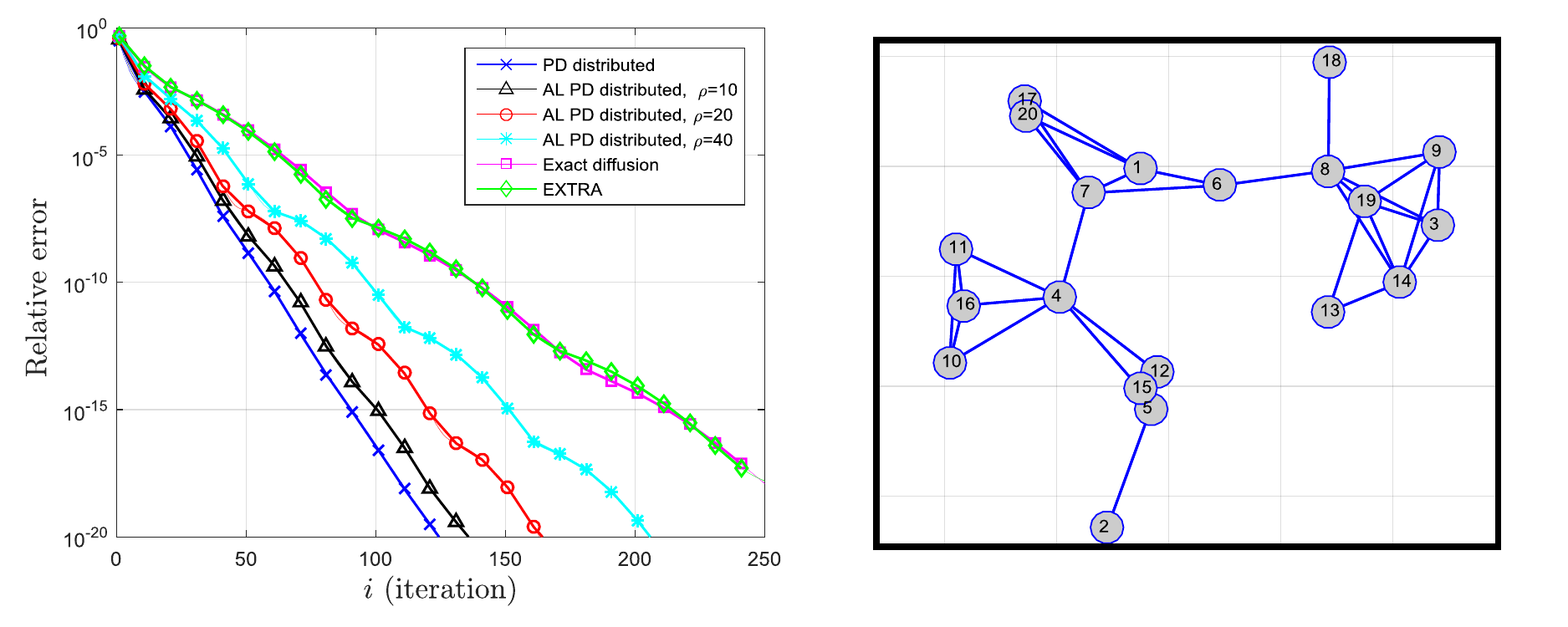}
	\vspace{-5mm} 
	\caption{ \small The left plot shows the simulation result for well conditioned local costs. The right plot shows the network topology used in the simulations. Relative error is $\|\sw_i-\sw^\star\|^2/\|\sw^\star\|^2$. }
\label{fig:simresults1}
\end{figure} 
\vspace{-3mm}
 \begin{figure}[h]
	\includegraphics[width=\linewidth]{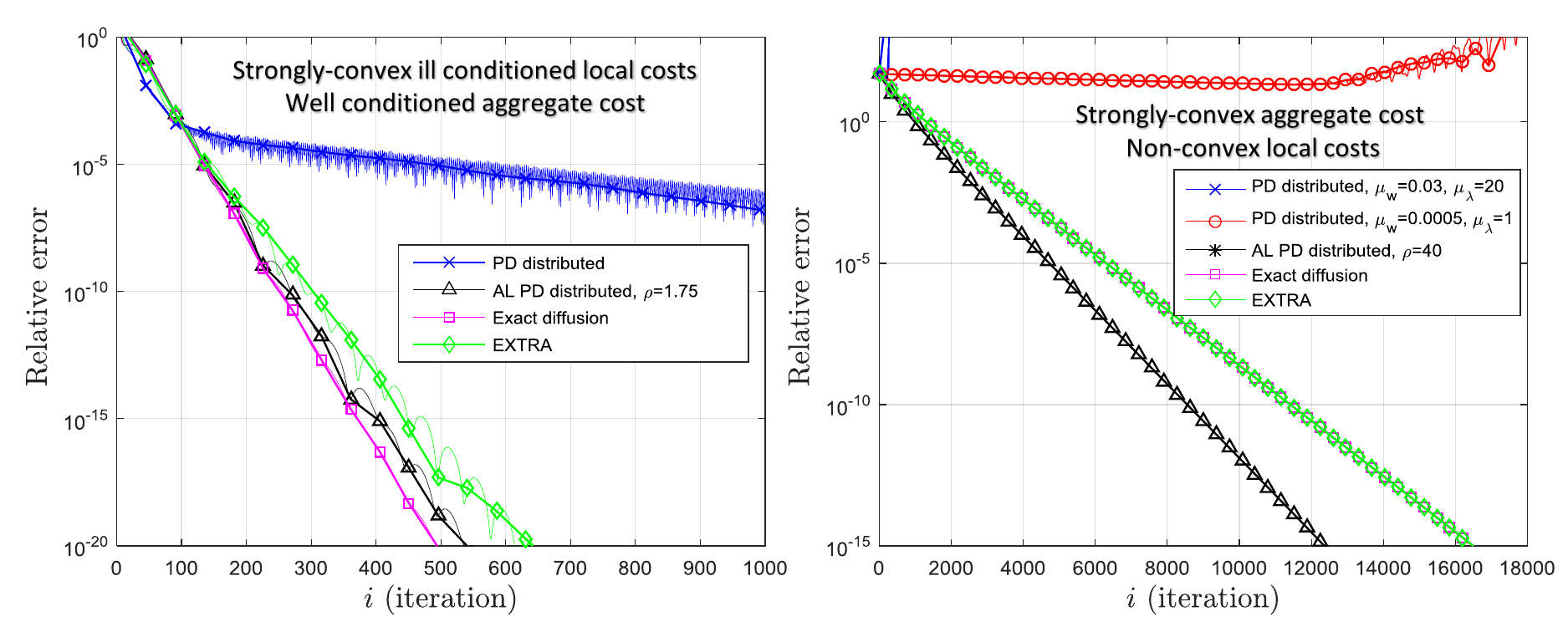}
	\vspace{-5mm}
	\caption{\small  Simulation result for ill-conditioned and non-convex local costs cases. }
\label{fig:simresults2}
\end{figure}
\section{Concluding Remarks} 
In this work, we studied the linear convergence of the classical incremental primal-dual gradient algorithm \eqref{alg_pd_da_inc}. We provided an original proof that is applicable to both the Lagrangian and augmented Lagrangian implementations. Moreover, we proved the linear convergence of the non-incremental implementation \eqref{alg_pd_da_non_inc} by relating it to the incremental one. Finally, we studied algorithm \eqref{alg_pd_da_inc} in distributed multi-agent optimization problems. The effect of the AL term on the performance of distributed algorithms is illustrated in theory and validated by means of simulation.}
\bibliographystyle{ieeetran}
\bibliography{myref_review,bibtex} 
\appendices
\section{Proof of Corollary \ref{Cor_noninc_convergence} } \label{appendix_corro_noninc}
 For $\eta=0$, we know from Lemma \ref{lemma:noninc_equivalnt} that recursion  \eqref{alg_pd_da_non_inc}  is equivalent to the incremental implementation with $\rho=\eta=-\mu_\lambda$, namely,
\begin{subequations}
\eq{
w_i&=w_{i-1}-\mu_w \big(\grad J'(w_{i-1})+ B\tran \lambda_{i-1}\big)  \\
\lambda_i &= \lambda_{i-1}+\mu_\lambda (Bw_i-b) 
}
 \end{subequations}
 where  $
J'(w)=J_{-\mu_\lambda}(w)=J(w)-{\mu_\lambda \over 2} \|Bw-b\|^2$. 
The above recursion is exactly \eqref{alg_pd_da_inc} with $\rho=0$ and cost $J'(w)$ instead of $J(w)$.  Therefore, its analysis follows from Theorem \ref{theorem1} as long as  $J'(w)$ is $\delta'-$smooth and $\nu'$-strongly-convex for some $\delta' \geq \nu' >0$.  It holds that:
\eq{
&\big(\grad J'(w_1)-\grad J'(w_2)\big)\tran (w_1-w_2) \nonumber \\
&=\big(\grad J(w_1)-\grad J(w_2)\big)\tran (w_1-w_2) - \mu_\lambda \|B(w_1-w_2)\|^2
\nonumber \\
 & \leq \|\grad J(w_1)-\grad J(w_2)\| \|w_1-w_2\| - \mu_\lambda \sigma^2_{\min} (B) \|w_1-w_2\|^2 \nonumber \\
 & \leq \big(\delta - \mu_\lambda \sigma^2_{\min} (B)\big) \|w_1-w_2\|^2, \quad \forall \ w_1,w_2 \in \real^M
}
where the first inequality holds from Cauchy-Schwartz  and  $\|B(w_1-w_2)\|^2 \geq \sigma^2_{\min} (B) \|w_1-w_2\|^2$. The last inequality holds since $J(w)$ is $\delta$-smooth.
The above inequality is equivalent to  the cost $J_{-\mu_\lambda}(w)$ being $\delta'=\delta - \mu_\lambda
\sigma^2_{\min}(B)$ smooth -- see \cite[Theorem 2.1.5]{nesterov2013introductory}. Moreover, from strong-convexity condition \eqref{stron-convexity}, it also holds that 
\eq{
&\big(\grad J'(w_1)-\grad J'(w_2)\big)\tran (w_1-w_2) \nonumber \\
&=\big(\grad J(w_1)-\grad J(w_2)\big)\tran (w_1-w_2) - \mu_\lambda \|B(w_1-w_2)\|^2
\nonumber \\
 & \geq  \nu \|w_1-w_2\|^2- \mu_\lambda \|B(w_1-w_2)\|^2 \nonumber \\
 & \geq  \big(\nu-\mu_\lambda
\sigma^2_{\max}(B)\big) \ \|w_1-w_2\|^2 , \quad \forall \ w_1,w_2 \in \real^M }
Hence, the cost $J'(w)=J_{-\mu_\lambda}(w)$ is $\nu'=\nu-\mu_\lambda
\sigma^2_{\max}(B)>0$ strongly-convex if $\mu_\lambda < \nu / \sigma^2_{\max}(B)$. By replacing $\delta$ and $\nu$ with $\delta'$ and $\nu'$ in \eqref{step_sizes} and setting $\rho=0$ we get conditions \eqref{step_sizes_noninc}.
\section{Proof of Corollary \ref{theorem2} } \label{appendix_corro_distributed}
Note that if $\lambda_{-1}=0$ and $\sy_{-1}=0$, then from \eqref{dual-ascent_dist_not} and \eqref{dual-ascent_dist} it holds that $\sy_i=\cB \lambda_i$ for all $i \geq -1$. Since $\lambda_i$ lies in the range space of $\cB$, it follows from Lemma \ref{lemma:lower_bound} that $\sy_i=0 \iff \lambda_i=0$. Thus, the primal iterates \eqref{primal-descent_dist_not} and \eqref{primal-descent_dist} are equivalent if $\lambda_{-1}=0$ and $\sy_{-1}=0$. Moreover, if recursion \eqref{primal-descent_dist_not}--\eqref{dual-ascent_dist_not} converges linearly to $(\sw^\star,\lambda^\star_b)$, then  recursion \eqref{primal-descent_dist}--\eqref{dual-ascent_dist} converges linearly to $(\sw^\star,\cB \lambda^\star_b)$ and its convergence properties follow from Theorem \ref{theorem1}. It remains to verify the conditions in Theorem \ref{theorem1} hold for the two cases $\rho=0$ and $\rho>0$.  For $\rho>0$, it holds that the cost $\cJ_\rho(\sw)= \cJ(\sw) + {\rho \over 2} \| \cB \sw\|^2$  is $\delta_\rho$-smooth  with $\delta_\rho=\delta+\rho \sigma^2_{\max}(\cB)$. Moreover, since the aggregate cost $\sum_{k=1}^K J_k(w): \real^M \rightarrow \real$ is $\bar{\beta}$-strongly-convex, it holds from \ref{lemma_penalized cost} that the augmented penalized cost $\cJ_\rho(\sw)$ is $\nu_{\rho}$-strongly convex with respect to $\sw^\star$. For $\rho=0$, the augmented cost $\cJ_0(\sw)=\cJ(\sw)= \sum_{k=1}^K J_k(w_k)$ is separable in $\{w_k\}$ so that $\grad \cJ_0(\sw)={\rm col}\{\grad J_k(w_k)\}_{k=1}^K$. Thus, $\cJ(\sw)$ is strongly-convex if, and only, if  each individual cost is strongly-convex. Since $J_k(w)$ is $\delta$-smooth and $\beta_k$-strongly-convex, it can be verified that $\cJ_0(\sw)$ is $\delta$-smooth and $\nu_0$-strongly-convex where $\nu_0=\min_k \beta_k$.    
 %\begin{algorithm}[t] 
%\caption*{\textrm{\bf{Algorithm}} (PD distributed algorithm \eqref{primal-descent_dist}--\eqref{dual-ascent_dist})}
%{\bf Setting}: Let $w_{k,-1}$ arbitrary, $y_{k,-1}=0$, $\rho \geq 0$, and $u_{k,-1}=w_{k,-1}-\sum_{s \in \cN_k} a_{sk}w_{s,-1}$.
%\begin{subequations}
%\label{pd_da_distributed}
%\eq{
%w_{k,i}&=w_{k,i-1}-\mu_w \big(\grad J_k(w_{k,i-1})+\rho u_{k,i-1}+  y_{k,i-1}\big) \label{p_dist_implemet} \\
%u_{k,i}&=w_{k,i}-\sum_{s \in \cN_k} a_{sk}w_{s,i} \quad { \rm {\bf (sharing \ step)}} \\ 
%y_{k,i} &= y_{k,i-1}+\mu_\lambda u_{k,i} \label{d_dist_implemet}
%}
% \end{subequations}
%\end{algorithm} 
\end{document}